\documentclass[12pt,reqno]{amsart}
\usepackage[a4paper,margin=2.5cm,top=2.5cm,bottom=2.5cm,centering,headheight=3ex,headsep=4ex,vcentering]{geometry}

\usepackage{amssymb,amsfonts,amsmath,amsthm}
\usepackage{orcidlink}

\usepackage[cal=cm,scr=euler]{mathalfa}
\usepackage{microtype}
\usepackage{mlmodern}

\newcommand{\bt}{\overline{bt}}

\theoremstyle{plain}

\numberwithin{equation}{section}

\newtheorem{theorem}{Theorem}[section]
\newtheorem{lemma}[theorem]{Lemma}

\newtheorem{proposition}[theorem]{Proposition}
\newtheorem{conjecture}[theorem]{Conjecture}
\theoremstyle{definition}

\usepackage{fancyhdr}

\title[Extending Recent Arithmetic Properties of Overcubic Partition Tuples]{Extending Recent Arithmetic Properties of Overcubic Partition Tuples}

\author[S. Maity]{Saikat Maity}
\address[S. Maity]{Mathematical and Physical Sciences division, School of Arts and Sciences, Ahmedabad University, Ahmedabad 380009, Gujarat, India}
\email{saikat.maity@ahduni.edu.in}

\author[M. P. Saikia]{Manjil P. Saikia}
\address[M. P. Saikia]{Mathematical and Physical Sciences division, School of Arts and Sciences, Ahmedabad University, Ahmedabad 380009, Gujarat, India}
\email{manjil.saikia@ahduni.edu.in}

\thanks{Corresponding Author: Manjil P. Saikia (manjil.saikia@ahduni.edu.in).}


\begin{document}

\begin{abstract}
 In recent years, several mathematicians have looked at a class of partitions called the overcubic partition $k$–tuples, for small values of $k$. They have proved several divisibility results satisfied by these partitions. We continue this study and prove a few infinite family of congruences. Our proofs use elementary techniques from $q$-series and number theory.
\end{abstract}

\maketitle

\vspace{5mm}

\section{Introduction}\label{sec:intro}

A partition $\pi$ of a positive integer $n$ is a weakly decreasing sequence of positive integers $\pi_1 \geq \pi_2 \geq \cdots \geq \pi_k > 0$ such that $\pi_1 + \pi_2 + \cdots + \pi_k = n.$ The $\pi_i$ are called the parts of the partition. The number of partitions of $n$, denoted $p(n)$, satisfies the celebrated Ramanujan congruences
\[
  p(5n+4)\equiv 0\pmod{5},\quad
  p(7n+5)\equiv 0\pmod{7},\quad
  p(11n+6)\equiv 0\pmod{11},
\]
which have inspired a rich literature on arithmetic properties of partition functions and their generalizations \cite{Andrews1998}. By convention, we take $p(0)=1$. One of the basic questions in the theory of integer partitions is whether the partition function $p(n)$, or any of its subclasses, satisfies interesting arithmetic properties. The aim of this paper is to look at one such subclass and prove some infinite family of congruence results.

Since the time of Euler, integer partitions have been studied extensively, leading to many interesting subclasses and generalizations. One such generalization is the \textit{cubic partition} function, introduced by Chan \cite{Chan1, Chan2}. A cubic partition of a positive integer $n$ is a partition in which the even parts may appear in two different colors. Let $a(n)$ denote the number of such partitions of $n$. The generating function for $a(n)$ is given by
\[
\sum_{n \geq 0} a(n) q^n = \frac{1}{f_1 f_2},
\]
where, we use the standard notation 
$f_k := (q^k; q^k)_\infty = \prod_{i\geq 1}(1-q^{ki})$ for $|q|<1$.

Shortly after Chan's work, Kim \cite{Kim} introduced an overpartition version of cubic partitions, known as the \textit{overcubic partition} function. Let us first define overpartitions. An \emph{overpartition} of $n$, introduced by Corteel and Lovejoy~\cite{Corteel2004}, is a partition of $n$ in which the first occurrence of each part may be overlined.  For example, the eight overpartitions of $3$ are
\[
  3, \bar{3}, 2+1, \bar{2}+1, 2+\bar{1}, \bar{2}+\bar{1}, 1+1+1, \bar{1}+1+1.
\]
The number of overpartitions of $n$, denoted $\overline{p}(n)$, has the  generating function \[\sum_{n\geq 0}\overline{p}(n)q^n=\frac{f_2}{f_1^2}.\]

In a similar vein, let $\overline{a}(n)$ denote the number of overcubic partitions of $n$. In this setting, Kim considers cubic partitions where the first occurrence of each part may be overlined. The generating function is given by
\begin{equation}\label{gf-an}
\sum_{n\geq 0} \overline{a}(n) q^n = \frac{f_4}{f_1^2 f_2}.
\end{equation}
The work of Chan \cite{Chan1, Chan2} and Kim \cite{Kim} has given rise to a lot of follow-up work by several mathematicians. Zhao and Zhong \cite{ZhaoZhong} studied \textit{cubic partition pairs}, denoted by $b(n)$. The generating function for $b(n)$ is given by
\[
\sum_{n\geq 0} b(n) q^n = \frac{1}{f_1^2 f_2^2}.
\]
Kim \cite{Kim2} then considered the overpartition analogue of this function. Let $\overline{b}(n)$ denote the number of overcubic partition pairs of $n$. Its generating function is
\[
\sum_{n\geq 0} \overline{b}(n) q^n = \frac{f_4^2}{f_1^4 f_2^2}.
\]

This idea can be extended further, and indeed recently, Nayaka, Dharmendra, and Kumar \cite{NayakaDharmendraKumar} introduced overcubic partition triples. They denoted by $\bt(n)$ the number of such partitions of $n$, with generating function
\begin{align}\label{eq-gf}
\sum_{n\geq 0} \bt(n) q^n = \frac{f_4^3}{f_1^6 f_2^3}.
\end{align}
In a recent work, Saikia and Sarma~\cite{SS25} defined the overcubic partition $k$-tuples function with the following generating function
\begin{equation} \label{eq:gf-b_k}
\sum_{n \geq 0} \overline{b}_k(n) q^n = \frac{f_4^k}{f_1^{2k} f_2^k}.
\end{equation}
Note that $k=1$ gives us $\bar a(n)$ back, $k=2$ gives us $\bar b(n)$, and $k=3$ gives us $\overline{bt}(n)$. We also note that Ramanujan-type congruences have been studied for $a(n)$, $\overline{a}(n)$, $b(n)$, and $\overline{b}(n)$, and more recently for $\bt(n)$ by Nayaka, Dharmendra, and Kumar \cite{NayakaDharmendraKumar}. 

Saikia and Sarma \cite{SS25} proved several results related to $\overline{bt}(n)$ and $\bar b_{2k+1}(n)$ for $k\geq 0$, which extended some results of Nayaka, Dharmendra, and Kumar \cite{NayakaDharmendraKumar}. In their paper, Saikia and Sarma \cite{SS25} gave some conjectures related to the $\overline{bt}(n)$ function, which were recently proved by Chen, Jin, and Yao~\cite{Chen-Zing-Yao}, and later extended by Tang~\cite{Tang-BMMS}. Other work on these classes of partitions include those by Sellers \cite{Sellers} and Shivashankar and Naika \cite{ShivashankarNaika} for $k=1$, Lin \cite{Lin}, Naika and Shivashankar \cite{NaikaShivashankar}, and Ray and Barman \cite{RayBarman} for $k=2$, Nayaka \cite{Nayaka2026} for $k=3$, Buragohain and Saikia \cite{BuragohainSaikia2024} for $k=4m+2, 4m+3, 8m+2, 8m+4, 16m+4$, and Chaćon and Sellers \cite{ChaconSellers} for $k=1,3,4,8,25j+8,27j+8, 27j+10, 27j+14, 9j+2, 9j+7, 3j+1, 2j+1$. The aim of this paper is to contribute to this growing list of works and prove some infinite family of congruences, given by the theorems that now follow.

\begin{theorem}\label{thm-4}
For all integers $n\geq 0$, $i\geq 3$, and $k\geq 0$, we have
\begin{align}
\overline{b}_{2^i k + 2^{i-1}}(4n+2) &\equiv 0 \pmod{2^{i+1}}, \label{eq:b_4n+2}\\
\overline{b}_{2^i k + 2^{i-1}}(4n+3) &\equiv 0 \pmod{2^{i+2}}. \label{eq:b_4n+3}
\end{align}
\end{theorem}

\begin{theorem}\label{thm-shift}
For all integers $n\geq 0$, $i\geq 3$, and $k\geq 0$, we have
\begin{equation}\label{eq:shift-4n+2}
  \overline{b}_{\,2^{i}k+2^{i-1}-2}(4n+2)\equiv 0 \pmod{2^{i+1}}.
\end{equation}
\end{theorem}

For one sub-arithmetic progression we can do better than \eqref{eq:b_4n+2}, as given below.
\begin{theorem}\label{thm-new}
    For all integers $n\geq 0$, $i\geq 3$, and $k\geq 0$, we have
    \begin{align}
        \bar b_{2^ik+2^{i-1}}(8n+5)&\equiv 0 \pmod{2^{i+1}},\label{eq:8n+5}\\
        \overline{b}_{2^i k + 2^{i-1}}(8n+6) &\equiv 0 \pmod{2^{i+3}}. \label{eq:8n+6}
    \end{align}
\end{theorem}

Similarly, for one sub-arithmetic progression we can do better than \eqref{eq:shift-4n+2}, as given below.
\begin{theorem}\label{thm-shift2}
For all integers $n\geq 0$, $i\geq 3$, and $k\geq 0$, we have
\begin{equation}\label{eq:shift-8n+6}
  \overline{b}_{\,2^{i}k+2^{i-1}-2}(8n+6)\equiv 0 \pmod{2^{i+3}}.
\end{equation}
\end{theorem}

Some remarks are in order:
\begin{enumerate}
\item Theorem \ref{thm-4} can be thought of as extensions of the following result of Buragohain and Saikia \cite{BuragohainSaikia2024} as observed by Chaćon and Sellers \cite[Theorem 5.2]{ChaconSellers}: for all $\alpha\geq 1, \ell \geq 0$ and $n\geq 1$, we have
\begin{equation}\label{eq:bs}
    \bar b_{2^{\alpha-1} \ell}(n) \equiv 0 \pmod{2^{\alpha}}.
\end{equation}
\eqref{eq:b_4n+2} improves the modulus by an extra factor of $2$, while \eqref{eq:b_4n+3} improves the modulus by an extra factor of $2^2$. Also note that \eqref{eq:bs} only gives a modulo $2$ result for the case considered in Theorem \ref{thm-shift}.
\item \eqref{eq:bs} already gives us, for all $m, k\geq 0, i\geq 3$:
\begin{equation}\label{eq:4n+1}
    \bar b_{2^ik+2^{i-1}}(4m+1)\equiv 0 \pmod{2^i}.
\end{equation} In \eqref{eq:8n+5} we improve the modulus by factor of $2$ for the case $m=2n$ in \eqref{eq:4n+1}. While in \eqref{eq:8n+6} we improve the modulus by a factor of $2^2$ compared to \eqref{eq:b_4n+2}. Similarly, in \eqref{eq:shift-8n+6} we improve the modulus by a factor of $2^2$ compared to \eqref{eq:shift-4n+2}.
    \item Theorem \ref{thm-4} partially extends one congruence proved in \cite{SS25}, where the authors proved the $i=k=1$ case of \eqref{eq:b_4n+3}. In addition, Theorem \ref{thm-4} also generalizes a few cases of \cite[Theorem 1.9]{SS25}. Moreover, the moduli in \eqref{eq:b_4n+2} and \eqref{eq:b_4n+3} are best possible.
    \item The authors strongly believe more infinite families are lying hidden. We note the following related conjecture, which we believe may follow from techniques similar to our own.
\begin{conjecture}\label{conj:main}
For all $n\geq 0, i\geq 3, k\geq 0$, we have
\begin{align}
\overline{b}_{2^i k + 2^{i-1}-1}(8n+5) &\equiv 0 \pmod{2^{i+2}}, \label{conj:b_8n+5}\\
\overline{b}_{2^i k + 2^{i-1}-1}(8n+7) &\equiv 0 \pmod{2^{i+3}} \label{conj:b_8n+7}.
\end{align}
\end{conjecture} It may also be worthwhile to check similar congruences modulo powers of other primes, we leave this avenue open for the interested readers.
\end{enumerate}

This paper is organized as follows: in Section \ref{sec:prel} we collect several results that will be used to furnish the proofs of the above results in Section \ref{sec:proof}. Our proof uses properties of Ramanujan's theta functions and some other elementary results from number theory.

\section{Preliminaries}\label{sec:prel}

Throughout this paper we fix integers $i\geq 3$ and $k\geq 0$ and write
\[
M := 2^i k + 2^{i-1} = 2^{i-1}(2k+1),
\]
so that $\nu_2(M)=i-1$ and, in particular, $M\equiv 0\pmod 4$. Here, $\nu_2(x)$ denotes the $2$-adic valuation of $x$, that is, the
largest power of $2$ dividing $x$.

Our proof rests on expressing the generating function \eqref{eq:gf-b_k}
in terms of Ramanujan's theta function
\[
\varphi(q) := \sum_{n=-\infty}^{\infty} q^{n^2} = 1 + 2\sum_{n\geq 1} q^{n^2},
\]
and then reducing modulo the relevant power of $2$. We will use the
classical formulas \cite[Eq.~(1.5.8)]{Hirschhorn}
\begin{equation}\label{eq:phi-products}
\varphi(-q) = \frac{f_1^2}{f_2}, \qquad\text{and}\qquad
\varphi(-q^2) = \frac{f_2^2}{f_4},
\end{equation}
as well as the identity \cite[Eq.~(1.5.13)]{Hirschhorn}
\begin{equation}\label{eq:phi-mult}
\varphi(q)\,\varphi(-q) = \varphi(-q^2)^2.
\end{equation}
We also need the following later \cite[Eq. (1.5.7)]{Hirschhorn}
\[
\psi(q):=\sum_{n\geq 0} q^{n(n+1)/2}=\frac{f_2^2}{f_1}.
\]

We need the following lemma connecting the two functions defined above, for the proof of Theorem \ref{thm-shift}.
\begin{lemma}\label{lem:shift-multiplier}
    We have
    \begin{equation}\label{eq:gamma-dissection}
  \varphi(-q)^2\,\varphi(-q^2)^2
  =\varphi(-q^4)^4-4q\,\psi(q^4)^2\,\varphi(-q^2)^2 .
\end{equation}
Writing $\displaystyle \varphi(-q)^2\,\varphi(-q^2)^2=\sum_{N\geq 0}\gamma(N)\,q^N$,
it follows that
\begin{equation}\label{eq:gamma-properties}
  \gamma(N)=0 \ \text{ for } N\equiv 2 \pmod 4,
  \qquad
  \gamma(N)\equiv 0 \pmod 4 \ \text{ for } N \text{ odd.}
\end{equation}
\end{lemma}
\begin{proof}
We need the following identity \cite[(34.1.13)]{Hirschhorn}
\[
\varphi(-q)^2=\varphi(q^2)^2-4q\,\psi(q^4)^2.
\]
Multiplying this by $\varphi(-q^2)^2$ on both sides and using \eqref{eq:phi-mult} with $q\mapsto
q^2$ we obtain \eqref{eq:gamma-dissection}.

    For the proof of \eqref{eq:gamma-properties}, notice that the RHS of \eqref{eq:gamma-dissection} modulo $4$ is a series in $q^4$, which immediately implies the result.
\end{proof}
\noindent The following lemma will be used in the proof of Theorem \ref{thm-shift2}.
\begin{lemma}\label{lem:pi-refined}
With $\gamma(N)$ as in
Lemma~\textup{\ref{lem:shift-multiplier}}, we have the following
\begin{align*}
   \gamma(0)&=1, \\
   \gamma(N)&\equiv 0 \pmod 8 \ \text{ for } N\equiv 0 \pmod 4,\ N\ge 4,\\
   \gamma(N)&\equiv 0 \pmod 4 \ \text{ for } N \text{ odd},\\
   \gamma(N)&\equiv 0 \pmod 8 \ \text{ for every odd $N\not\equiv 1 \pmod 8$}.
\end{align*}
\end{lemma}

\begin{proof}
The first assertion is easy to see. The third assertion is just a repetition of the second assertion in \eqref{eq:gamma-properties}.

    By \eqref{eq:gamma-dissection}, for $N\equiv 0\pmod 4$, say $N=4m$,
\[
  \gamma(4m)=[q^{4m}]\,\varphi(-q^4)^4=[q^{m}]\,\varphi(-q)^4=(-1)^m\,[q^{m}]\,\varphi(q)^4,
\]
where the operator $[q^X]$ extracts out the coefficient of $q^X$ from the power series on which it is applied to. Now, by Jacobi's four-square theorem \cite[Theorem 65]{Johnson}, we know that $[q^{m}]\varphi(q)^4=8\sum_{d| m,\,4\nmid d}d$,
which equals $1$ for $m=0$ and is divisible by $8$ for $m\ge 1$. This gives the
second assertion.

For $N$ odd, \eqref{eq:gamma-dissection} gives
$\gamma(N)=-4\,[q^{N-1}]\bigl(\psi(q^4)^2\,\varphi(-q^2)^2\bigr)$, so $4|\gamma(N)$, as already observed. Reducing the series modulo $2$, we have $\varphi(-q^2)\equiv 1\pmod 2$ from definition and $\psi(q^4)^2\equiv\psi(q^8)\pmod 2$, which is easy to see. Hence,
$\psi(q^4)^2\varphi(-q^2)^2\equiv\psi(q^8)\pmod 2$. Thus, for $N$ odd we have
\[
\gamma(N)\equiv -4[q^{N-1}]\psi(q^8) \pmod 8,
\]
where the RHS is a series in $q^8$, from which the final assertion now follows.
\end{proof}

Let us now rewrite \eqref{eq-gf} in terms of $\varphi(q)$'s.
\begin{lemma}\label{lem:gf-theta}\cite[Corollary 3.2]{ChaconSellers}
We have
\begin{equation}\label{eq:gf-theta}
\sum_{n\geq 0}\overline{b}_M(n)\, q^n
= \frac{1}{\varphi(-q)^M\,\varphi(-q^2)^M}
= \varphi(q)^M\,\varphi(q^2)^{3M}\,\varphi(q^4)^{6M}\,\varphi(q^8)^{12M}\cdots,
\end{equation}
where the infinite product is $\prod\limits_{r\geq 0}\varphi\big(q^{2^r}\big)^{c_r}$
with $c_0 = M$ and $c_r = 3\cdot 2^{r-1}M$ for $r\geq 1$.
\end{lemma}

We also need the following two elementary facts.
\begin{lemma}\label{lem:dbinom-val}
Let $e$ be a positive integer and $m\geq 1$. Then
\[
\nu_2\dbinom{e}{m} \;\geq\; \nu_2(e) - \nu_2(m).
\]
Consequently, writing $\varphi(q^c) = 1 + 2T_c$ with
$T_c := \sum_{n\geq 1} q^{cn^2}$, the binomial expansion
\begin{equation}\label{eq:phi-expansion}
\varphi(q^c)^{e} = \sum_{m=0}^{e}\dbinom{e}{m} 2^m\, T_c^{\,m}
\end{equation}
has $m$-th term of $2$-adic valuation at least $m + \nu_2(e) - \nu_2(m)$.
\end{lemma}

\begin{lemma}\label{lem:phi-power}
For every integer $a\geq 0$ and every $c\geq 1$,
\[
\varphi(q^c)^{2^a u} \equiv 1 \pmod{2^{a+1}}
\]
for all positive integers $u$.
\end{lemma}

\begin{proof}
By Lemma \ref{lem:dbinom-val} with $e = 2^a u$, the $m$-th term
of \eqref{eq:phi-expansion} has valuation at least
$m + a - \nu_2(m) \geq a+1$ for all $m\geq 1$, since $m > \nu_2(m)$.
\end{proof}

We will also need two simple counting facts about representations by
sums of squares. For $N, m \geq 1$ let
\[
r_m(N) := \#\big\{(a_1,\ldots,a_m)\in\mathbb{Z}_{\geq 1}^m :
a_1^2+\cdots+a_m^2 = N\big\}
\]
denote the number of \emph{ordered} representations of $N$ as a sum of
$m$ positive squares, so that \[T_1^{\,m} := \left(\sum_{n\geq 1}q^{n^2}\right)^m= \sum_{N\geq 1} r_m(N)q^N.\]

\begin{lemma}\label{lem:rep-counts}
Let $N\geq 1$.
\begin{enumerate}
\item[(a)] If $N\equiv 3\pmod 4$, then $r_1(N) = r_2(N) = 0$, and $4| r_4(N)$.
\item[(b)] If $N\equiv 2\pmod 4$, then
$r_2(N) \equiv \delta(N) \pmod 2$, where $\delta(N) = 1$ if
$N = 2a^2$ for some odd $a\geq 1$, and $\delta(N)=0$ otherwise.
\item[(c)] If $N \equiv 6 \pmod 8$, then $r_1(N) = r_2(N) = 0$, and $4|r_4(N)$. Moreover, 
\[
r_3(N) \equiv r_3(N/2) \pmod 2.
\]
\end{enumerate}
\end{lemma}

\begin{proof}
Squares are congruent to $0$ or $1$ modulo $4$. Hence $a^2 \equiv 3$ and $a^2+b^2\equiv 3 \pmod 4$ are
impossible, giving $r_1(N)=r_2(N)=0$ in case (a). 

If
$a_1^2+a_2^2+a_3^2+a_4^2 \equiv 3 \pmod 4$, then exactly three of the
$a_j$'s are odd. We divide the solutions counted by $r_4(N)$ into four
classes according to the position of the unique even coordinate;
permuting that position gives bijections between the four classes, so
$r_4(N) = 4\,r$ for some integer $r$, proving (a).

For (b), $a^2+b^2\equiv 2\pmod 4$ forces both $a$ and $b$ to be odd. The
involution $(a,b)\mapsto(b,a)$ pairs up the solutions with $a\neq b$,
leaving at most the single fixed point $a=b$, which occurs precisely
when $N = 2a^2$ with $a$ odd.

We now prove (c), assume that $N\equiv 6 \pmod 8$. Since every square is congruent to $0$, $1$, or $4$ modulo $8$, the congruence $x^2 \equiv 6\pmod8$ has no solution, and neither does $x^2+y^2 \equiv 6\pmod8$, because the pairwise sums of $\{0,1,4\}$ modulo $8$ are $0,1,2,4,5$. Hence $r_1(N)=r_2(N)=0.$

Now let $a_1^2+a_2^2+a_3^2+a_4^2=N$. Since $N\equiv6\pmod8$, the only possible residue pattern of $(a_1^2,a_2^2,a_3^2,a_4^2)$ modulo $8$ is $\{0,1,1,4\}$ since the sum has to sum up to $6$ in this case. Thus,  exactly one coordinate is divisible by $4$, exactly
one is congruent to $2$ modulo $4$, and the remaining two coordinates are odd. Notice that we have a unique multiple of $4$, so permuting the positions w.r.t to this coordinate gives us four equinumerous subsets inside the count of $r_4(N)$ which proves $4|r_4(N)$.

For the final part, we are interested in looking at $T_1^3$. Notice that
\[
T_1=\sum_{n\geq 1}q^{(2n)^2}+\sum_{n\geq 0}q^{(2n+1)^2}=\sum_{n\geq 1}q^{4n^2}+\sum_{n\geq 0}q^{8\binom{n+1}{2}+1}.
\]
Let us call the two summands $E$ and $O$ respectively to denote the even and odd exponents. Also notice that $O=q\psi(q^8)$, that is $O$ has exponents $\equiv 1 \pmod 8$ and $E$ has exponents $\equiv 0,4 \pmod 8$. We further have
\[
T_1^3=E^3+O^3+3E^2O+3O^2E,
\]
where the coefficients of the summands have exponents $\equiv$ to $\{0,4\}$, $\{3\}$, $\{1,5\}$ and $\{2,6\} \pmod 8$ respectively from left to right.

We set $N/2=M$ and already we have assumed $N\equiv 6 \pmod 8$, so $M\equiv 3 \pmod 8$. Thus,
\[
r_3(N)=3[q^N]O^2E=3[q^N]O^2q^4\psi(q^{32}),
\]
where in the first equality the contributions are only coming from the odd exponents of $E$ and the second equality is just \[\sum_{n\geq 1, n \text{ odd}}q^{4n^2}=\sum_{m\geq 0}q^{4(2m+1)^2}=q^4\sum_{m\geq 0}q^{32\binom{m+1}{2}}=q^4\psi(q^{32}),\] and
\[
r_3(M)=[q^M]O^3.
\]
We now use $O=q\psi(q^8)$ in the above to get
\[
r_3(N)=3[q^{N-6}]\psi(q^8)^2\psi(q^{32})=3[q^X]\psi(q)^2\psi(q^4), \quad \text{with }X:=\frac{N-6}{8}
\]
and
\[
r_3(M)=[q^{M-3}]\psi(q^8)^3=[q^Y]\psi(q)^3, \quad \text{with }Y:=\frac{M-3}{8}.
\]
Since $M=N/2$, we also have $X=2Y$.

Now, we have
\begin{align*}
  r_3(N)-r_3(M)  &=3[q^{2Y}]\psi(q)^2\psi(q^4)-[q^Y]\psi(q)^3 \\ &\equiv 3[q^{2Y}]\psi(q^2)\psi(q^4)-[q^Y]\psi(q)^3 \pmod 2 \\ &=  3[q^Y]\psi(q)\psi(q^2)-[q^Y]\psi(q)^3 \\
    &= 3[q^Y]\psi(q)\psi(q^2)-[q^Y]\psi(q)\bigl(\psi(q)^2\bigl) \\
    &\equiv 3[q^Y]\psi(q)\psi(q^2)-[q^Y]\psi(q)\bigl(\psi(q^2)\bigl) \pmod 2 \\
    &\equiv 0 \pmod 2,
\end{align*}
as required.
\end{proof}

Let us now prove the following important result for our proofs.
\begin{proposition}\label{prop:reduction(i+3)}
With $M = 2^{i-1}(2k+1)$ and $i\geq 3$, we have
\begin{equation}\label{eq:four-factor}
\sum_{n\geq 0}\overline{b}_M(n) q^n
\equiv \varphi(q)^{M} \varphi(q^2)^{3M} \varphi(q^4)^{6M}  \varphi(q^8)^{12M}
\pmod{2^{i+3}},
\end{equation}
Moreover,
\begin{align}\label{eq:final-expansion(i+3)}
\sum_{n\geq 0}\overline{b}_M(n) q^n &\equiv
1 + 2MT_1 + 4\binom{M}{2}T_1^2 + 8\binom{M}{3}T_1^3 + 16\binom{M}{4}T_1^4 \notag\\
&\quad +6MT_2+4\binom{3M}{2}T_2^2+8\binom{3M}{3}T_2^3+16\binom{3M}{4}T_2^4 \notag\\
&\quad +12MT_4+4\binom{6M}{2}T_4^2+24MT_8+4\binom{12M}{2}T_8^2\pmod{2^{i+3}},
\end{align}
where $T_c=\sum\limits_{n\ge1}q^{cn^2}$.

Furthermore,
\begin{equation}\label{eq:val-table(i+3)}
\begin{aligned}
\nu_2(2M)
&=\nu_2(6M)
=\nu_2 \left(4\binom{M}{2}\right)
=\nu_2 \left(4\binom{3M}{2}\right)
=i,\\[2mm]
\nu_2(8M)
&=\nu_2 \left(8\binom{M}{3}\right)
=\nu_2 \left(8\binom{3M}{3}\right)
=\nu_2(24M)
=i+2,\\[2mm]
\nu_2(12M)
&=\nu_2 \left(16\binom{M}{4}\right)
=\nu_2 \left(16\binom{3M}{4}\right)
=\nu_2 \left(4\binom{6M}{2}\right)
=\nu_2 \left(4\binom{12M}{2}\right)
=i+1.
\end{aligned}
\end{equation}
\end{proposition}

\begin{proof}
From Lemma~\ref{lem:gf-theta},
\[
\sum_{n\geq0}\overline{b}_M(n)q^n = \prod_{r\geq0}\varphi(q^{2^r})^{c_r},
\]
where
\[
c_0=M, \qquad c_r=3\cdot2^{r-1} M \quad(r\geq 1).
\]
Since
\(
\nu_2(M)=i-1,
\)
we have
\[
\nu_2(c_0)=\nu_2(c_1)=i-1, \qquad \nu_2(c_2)=i, \qquad \nu_2(c_3)=i+1,
\]
and for $r\geq 4$,
\[
\nu_2(c_r)=r+i-2\geq i+2.
\]
Hence, by Lemma~\ref{lem:phi-power}, we have
\[
\varphi(q^{2^r})^{c_r} \equiv 1 \pmod{2^{i+3}} \qquad(r\ge4),
\]
which proves \eqref{eq:four-factor}. 

It remains to truncate the four surviving factors in
\eqref{eq:four-factor}. 
By Lemma~\ref{lem:dbinom-val}, the $m$-th summand has $2$-adic valuation at least
\[
m+\nu_2(e)-\nu_2(m) = \nu_2(e)+m-\nu_2(m).
\]
Observe that
\begin{equation}\label{eq:m-nu(i+3)}
m-\nu_2(m) =
\begin{cases}
1, & m=1,2,\\
3, & m=3,\\
2, & m=4,\\
\geq 5, & m\geq 5.
\end{cases}
\end{equation}

For the factors $\varphi(q)^M$ and $\varphi(q^2)^{3M}$, we have $\nu_2(M)=\nu_2(3M)=i-1.$ Hence, by  Lemma~\ref{lem:dbinom-val} and \eqref{eq:m-nu(i+3)}, the $m$-th binomial term has valuation at least $(i-1)+m-\nu_2(m)$. Therefore the terms corresponding to $m=1,2,3,4$ have valuations $i,i,i+2,i+1$, respectively, while every term with $m\ge5$ has valuation at least $(i-1)+5=i+4$. Consequently, modulo $2^{i+3}$, only the terms with $m=0,1,2,3,4$ contribute.

Next consider the factor $\varphi(q^4)^{6M}$. Since $\nu_2(6M)=i$, the $m$-th binomial term has valuation at least $i+m-\nu_2(m)$. It follows from \eqref{eq:m-nu(i+3)} that the terms with $m=1,2$ have valuation $i+1$, whereas every term with $m\geq3$ has valuation at least $i+3$. Hence, modulo $2^{i+3}$, only the terms with $m=0,1,2$ remain.

Finally, for the factor $\varphi(q^8)^{12M}$, we have $\nu_2(12M)=i+1$.
Therefore the $m$-th binomial term has valuation at least $(i+1)+m-\nu_2(m)$. The terms with $m=1,2$ both have valuation $i+2$,
while every term with $m\geq3$ has valuation at least $i+4$.
Thus, modulo $2^{i+3}$, only the terms with $m=0,1,2$ contribute.

Consequently, we have
\begin{align*}
\sum_{n\geq 0}\overline{b}_M(n)\,q^n
&\equiv \left(1+2MT_1+4\binom M2T_1^2+8\binom M3T_1^3+16\binom M4T_1^4\right)\\
&\quad \times \left(1+6MT_2+4\binom{3M}{2}T_2^2+8\binom{3M}{3}T_2^3+16\binom{3M}{4}T_2^4\right)\\
&\quad \times \left(1+12MT_4+4\binom{6M}{2}T_4^2\right) \times \left(1+24MT_8+4\binom{12M}{2}T_8^2\right)\pmod{2^{i+3}}.
\end{align*}
Since every product of two nonconstant terms has valuation at least
\[
i+i=2i\geq i+3
\qquad(i\geq3),
\]
so all such mixed terms vanish modulo $2^{i+3}$.
This immediately yields \eqref{eq:final-expansion(i+3)}.

We now verify \eqref{eq:val-table(i+3)}. Write $M=2^{i-1} u$, where $u=2k+1$ is odd.
Since $M \equiv 0 \pmod4$, we have
\[
\nu_2(M-2) = \nu_2(3M-2) = \nu_2(6M-2) = \nu_2(12M-2) =1.
\]
It follows that
\[
\nu_2 \left(\binom{M}{2}\right) = \nu_2 \left(\binom{3M}{2}\right) =i-2,
\]
\[
\nu_2 \left(\binom{M}{3}\right) = \nu_2 \left(\binom{3M}{3}\right) =i-1,
\]
\[
\nu_2 \left(\binom{M}{4}\right) = \nu_2 \left(\binom{3M}{4}\right) =i-3,
\]
and
\[
\nu_2 \left(\binom{6M}{2}\right) = \nu_2 \left(\binom{12M}{2}\right) =i-1.
\]
Combining these identities with
\[
\nu_2(2)=1,\qquad
\nu_2(4)=2,\qquad
\nu_2(8)=3,\qquad
\nu_2(12)=2,\qquad
\nu_2(16)=4,\qquad
\nu_2(24)=3,
\]
immediately yields \eqref{eq:val-table(i+3)}.
\end{proof}

The following result is worth noting separately, where the proof follows from Lemma \ref{lem:phi-power} and Proposition \ref{prop:reduction(i+3)}.
\begin{proposition}\label{prop:reduction}
With $M = 2^{i-1}(2k+1)$ and $i\geq 3$, we have
\begin{equation}\label{eq:three-factor}
\sum_{n\geq 0}\overline{b}_M(n)\,q^n
\equiv \varphi(q)^{M}\,\varphi(q^2)^{3M}\,\varphi(q^4)^{6M}
\pmod{2^{i+2}},
\end{equation}
and, more precisely,
\begin{align}\label{eq:final-expansion}
\sum_{n\geq 0}\overline{b}_M(n)\,q^n
&\equiv 1
+ 2M\,T_1 + 4\dbinom{M}{2}T_1^2 + 16\dbinom{M}{4}T_1^4 \notag\\
&\quad + 6M\,T_2 + 4\dbinom{3M}{2}T_2^2 + 16\dbinom{3M}{4}T_2^4 \notag\\
&\quad + 12M\,T_4 + 4\dbinom{6M}{2}T_4^2
\pmod{2^{i+2}},
\end{align}
where $T_c = \sum_{n\geq 1}q^{cn^2}$. 

The valuations of the
coefficients in \eqref{eq:final-expansion} are
\begin{equation}\label{eq:val-table}
\begin{aligned}
&\nu_2(2M) = \nu_2(6M) = \nu_2\Big(4\dbinom{M}{2}\Big)
 = \nu_2\Big(4\dbinom{3M}{2}\Big) = i,\\
&\nu_2(12M) = \nu_2\Big(16\dbinom{M}{4}\Big)
 = \nu_2\Big(16\dbinom{3M}{4}\Big) = \nu_2\Big(4\dbinom{6M}{2}\Big) = i+1.
\end{aligned}
\end{equation}
\end{proposition}

\section{Proofs of the Main Theorems}\label{sec:proof}

\subsection{Proof of Theorem \ref{thm-4}}

\begin{proof}[Proof of \eqref{eq:b_4n+3}]
We extract the coefficient of $q^N$ with $N\equiv 3\pmod 4$ in
\eqref{eq:final-expansion}. The series $T_2$, $T_2^2$, $T_2^4$, $T_4$,
and $T_4^2$ have only even exponents and contribute nothing.
By Lemma \ref{lem:rep-counts}(a), $T_1$ and $T_1^2$ do not have any exponent in the residue class $3\pmod 4$, so they also contribute nothing.
The only remaining term is $16\dbinom{M}{4}T_1^4$, whose coefficient of
$q^N$ is $16\dbinom{M}{4}\,r_4(N)$. By \eqref{eq:val-table} and
Lemma \ref{lem:rep-counts}(a), we have
\[
\nu_2\Big(16\dbinom{M}{4}\,r_4(N)\Big) \geq (i+1) + 2 = i+3 \geq i+2,
\]
so the coefficient of $q^N$ vanishes modulo $2^{i+2}$, proving
\eqref{eq:b_4n+3}.
\end{proof}

\begin{proof}[Proof of \eqref{eq:b_4n+2}]
We now work modulo $2^{i+1}$, so by \eqref{eq:val-table} and \eqref{eq:final-expansion}, we have
\begin{equation}\label{eq:mod-2-i+1}
\sum_{n\geq 0}\overline{b}_M(n)\,q^n
\equiv 1 + 2M\,T_1 + 4\dbinom{M}{2}T_1^2 + 6M\,T_2 + 4\dbinom{3M}{2}T_2^2
\pmod{2^{i+1}}.
\end{equation}
We fix $N\equiv 2\pmod 4$ and consider each term in turn.

The series $T_1$ has exponents of the form $N\equiv 0,1\pmod 4$, so this contributes
nothing. For $T_2^2$, the coefficient of $q^N$ counts ordered pairs
$(a,b)$ with $2(a^2+b^2)=N$. That is, $a^2+b^2 = N/2$ is odd, so exactly
one of $a,b$ is even and the involution $(a,b)\mapsto(b,a)$ is
fixed-point free. Hence the coefficient is even and the term
$4\dbinom{3M}{2}T_2^2$ contributes a multiple of $2^{i+1}$.

We now look at $4\dbinom{M}{2}T_1^2 + 6M\,T_2$. The
coefficient of $q^N$ in $T_2$ is $1$ if $N = 2a^2$ for some (necessarily
odd, since $N\equiv 2\pmod 4$) integer $a\geq 1$, and $0$ otherwise;
that is, it equals $\delta(N)$ in the notation of Lemma
\ref{lem:rep-counts}(b). By Lemma
\ref{lem:rep-counts}(b) the coefficient of $q^N$ in
$T_1^2$ is $r_2(N) = \delta(N) + 2\varepsilon$ for some integer
$\varepsilon\geq 0$. Therefore the coefficient of $q^N$ in
$4\dbinom{M}{2}T_1^2 + 6M\,T_2$ equals
\[
4\dbinom{M}{2}\big(\delta(N)+2\varepsilon\big) + 6M\,\delta(N)
\equiv \Big(2M(M-1) + 6M\Big)\delta(N)
= 2M(M+2)\,\delta(N)
\pmod{2^{i+1}}.
\]
Finally, since $M\equiv 0\pmod 4$
we have $\nu_2(M+2)=1$, whence
\[
\nu_2\big(2M(M+2)\big) = 1 + (i-1) + 1 = i+1,
\]
so this contribution also vanishes modulo $2^{i+1}$. 

Summing everything up, the
coefficient of $q^N$ on the right-hand side of \eqref{eq:mod-2-i+1} is
divisible by $2^{i+1}$ for every $N\equiv 2\pmod 4$, proving
\eqref{eq:b_4n+2}.
\end{proof}

\subsection{Proof of Theorem \ref{thm-shift}}\label{sec:thm-shift}

We notice that $\dfrac{1}{\varphi(-q)\,\varphi(-q^2)}=\dfrac{f_4}{f_1^2 f_2}$, so by
\eqref{eq:gf-b_k} we have
$\sum_{n\geq 0}\overline{b}_k(n)\,q^n=\bigl(\varphi(-q)\varphi(-q^2)\bigr)^{-k}$
for every $k\geq 1$. Consequently
\begin{equation}\label{eq:gf-shift}
  \sum_{n\geq 0}\overline{b}_{M-2}(n)\,q^n
  \;=\;\varphi(-q)^2\,\varphi(-q^2)^2\sum_{n\geq 0}\overline{b}_{M}(n)\,q^n.
\end{equation}
With $\gamma(N)$ as in Lemma \ref{lem:shift-multiplier}, we can thus write 
\[
  \overline{b}_{M-2}(N)=\sum_{j=0}^{N}\gamma(j)\,\overline{b}_{M}(N-j).
\]
We fix $N\equiv 2 \pmod 4$ and split the sum according to $j\bmod 4$. By the first
part of \eqref{eq:gamma-properties} every term with $j\equiv 2\pmod 4$ vanishes,
so we have the following three summands
\[
  \overline{b}_{M-2}(N)
  =\underbrace{\sum_{j\equiv 0\,(4)}\!\!\gamma(j)\,\overline{b}_{M}(N-j)}_{=:S_0}
  +\underbrace{\sum_{j\equiv 1\,(4)}\!\!\gamma(j)\,\overline{b}_{M}(N-j)}_{=:S_1}
  +\underbrace{\sum_{j\equiv 3\,(4)}\!\!\gamma(j)\,\overline{b}_{M}(N-j)}_{=:S_3}.
\]

Let us now look at the three summands one at a time. First, for $S_0$ we note that $N-j\equiv 2\pmod 4$, so $\overline{b}_M(N-j)\equiv 0
\pmod{2^{i+1}}$ by \eqref{eq:b_4n+2}, hence $S_0$ is $\equiv 0\pmod{2^{i+1}}$. For $S_3$, we have $N-j\equiv 3\pmod 4$, so $\overline{b}_M(N-j)\equiv 0
\pmod{2^{i+2}}$ by \eqref{eq:b_4n+3}, hence $S_3\equiv 0\pmod{2^{i+2}}$. We need slightly more work on $S_1$.

For $S_1$, $j$ is odd, so $\gamma(j)\equiv 0\pmod 4$ by the
second part of \eqref{eq:gamma-properties}. Moreover $N-j\equiv 1\pmod 4$, and in
\eqref{eq:final-expansion} the coefficient of $q^{4m+1}$ receives contributions
only from $2M\,T_1$, $4\dbinom{M}{2}T_1^2$ and $16\dbinom{M}{4}T_1^4$. By
\eqref{eq:val-table} these three coefficients have $2$-adic valuation at least
$i$, so $\overline{b}_M(4m+1)\equiv 0\pmod{2^{i}}$ for all $m\geq 0$\footnote{This also follows from \eqref{eq:bs} but we choose to give an independent proof.}. Hence each
summand of $S_1$ has valuation at least $2+i\geq i+1$, and $S_1\equiv 0
\pmod{2^{i+1}}$.

This completes our proof.

\subsection{Proof of Theorem \ref{thm-new}}

\begin{proof}[Proof of \eqref{eq:8n+5}]
    In \eqref{eq:final-expansion} the only series contributing to exponents
$\equiv 1\pmod 4$ are $T_1$, $T_1^2$ and $T_1^4$. Since no square is
$\equiv 5\pmod 8$, the series $T_1$ contributes nothing to exponents
$\equiv 5\pmod 8$; thus for $N\equiv 5\pmod 8$,
\[
  \overline{b}_M(N)\equiv 4\binom{M}{2} r_2(N)+16\binom{M}{4} r_4(N)
  \pmod{2^{\,i+2}} .
\]
By \eqref{eq:val-table}, $\nu_2\!\bigl(4\binom{M}{2}\bigr)=i$ and
$\nu_2\!\bigl(16\binom{M}{4}\bigr)=i+1$. Moreover $r_2(N)$ is even: as $N$ is odd here,
$N=2a^2$ is impossible, so the involution $(a,b)\mapsto(b,a)$ on the positive
representations $a^2+b^2=N$ is fixed-point free. Hence both terms have valuation
at least $i+1$, proving the required congruence.
\end{proof}

\begin{proof}[Proof of \eqref{eq:8n+6}]
Let $N\equiv6\pmod8$. Extracting the coefficient of $q^N$ from Proposition~\ref{prop:reduction(i+3)}, we obtain
\begin{align} \label{eq:coeff-q^6mod8}
[q^N]\sum_{n\ge0}\overline{b}_M(n)q^n
&\equiv
2Mr_1(N)+4\binom{M}{2}r_2(N)+8\binom{M}{3}r_3(N)+16\binom{M}{4}r_4(N)\notag\\
&\quad
+6Mr_1(N/2)+4\binom{3M}{2}r_2(N/2)+8\binom{3M}{3}r_3(N/2)\notag\\
&\quad
+16\binom{3M}{4}r_4(N/2) \pmod{2^{i+3}},
\end{align}
since the series $T_4$, $T_4^2$, $T_8$, and $T_8^2$ contribute only powers of $q$ divisible by $4$ or $8$.

By Lemma~\ref{lem:rep-counts}(c),
\[
r_1(N)=r_2(N)=0,
\qquad
4| r_4(N),
\]
while $N/2 \equiv 3\pmod4$, so Lemma~\ref{lem:rep-counts}(a) gives
\[
r_1(N/2)=r_2(N/2)=0,
\qquad
4| r_4(N/2).
\]
Moreover,
\[
\nu_2 \left(16\binom{M}{4}\right) = \nu_2 \left(16\binom{3M}{4}\right) = i+1
\]
by \eqref{eq:val-table(i+3)}. Since $4|r_4(N)$ and $4|r_4(N/2)$, we have
\[
16\binom{M}{4}r_4(N) \equiv 16\binom{3M}{4}r_4(N/2) \equiv0 \pmod{2^{i+3}}.
\]
Finally,
\[
\nu_2 \left(8\binom{M}{3}\right) = \nu_2 \left(8\binom{3M}{3}\right) = i+2.
\]
Hence there exist odd integers $u$ and $v$ such that
\[
8\binom{M}{3}r_3(N)+8\binom{3M}{3}r_3(N/2) = 2^{i+2}\bigl(ur_3(N)+vr_3(N/2)\bigr).
\]
Since $u$ and $v$ are odd and by Lemma~\ref{lem:rep-counts}(c), we have
\[
r_3(N)\equiv r_3(N/2)\pmod2
\]
so
\[
8\binom{M}{3}r_3(N)+8\binom{3M}{3}r_3(N/2) \equiv 0 \pmod{2^{i+3}}.
\]

Combining all of the above, the R.H.S. of \eqref{eq:coeff-q^6mod8} vanishes modulo $2^{i+3}$ for every $N\equiv 6 \pmod8$, proving \eqref{eq:8n+6}.

\end{proof}

\subsection{Proof of Theorem \ref{thm-shift2}}
Recall from the proof of Theorem \ref{thm-shift}
\[
  \overline{b}_{M-2}(N)=\sum_{j=0}^{N}\gamma(j)\,\overline{b}_{M}(N-j),
\]
where by Lemma~\ref{lem:shift-multiplier} only the terms with
$j\not\equiv 2\pmod 4$ survive. Let us now fix $N=8n+6$ and we group the sum by the residue of
$j$ modulo $8$ (here the residues modulo $2$ and $6$ are absent). For each surviving
residue we bound the valuation
$\nu_2(\gamma(j))+\nu_2\!\bigl(\overline{b}_M(N-j)\bigr)$, using
Lemma~\ref{lem:pi-refined} together with
\eqref{eq:b_4n+2}, \eqref{eq:b_4n+3}, \eqref{eq:4n+1}, \eqref{eq:8n+5} and  \eqref{eq:8n+6}. We do this in the table below:
\[
\renewcommand{\arraystretch}{1.25}
\begin{array}{c|c|c|c|l}
 j \bmod 8 & \nu_2(\gamma(j)) & N-j \bmod 8 & \nu_2(\overline{b}_M(N-j)) & \text{sum} \\\hline
 0\ (j=0)   & 0      & 6 & \ge i+3\ \text{by \eqref{eq:8n+6}} & \ge i+3 \\
 0\ (j\ge8) & \ge 3  & 6 & \ge i+1\ \text{by \eqref{eq:b_4n+2}} & \ge i+4 \\
 4          & \ge 3  & 2 & \ge i+1\ \text{by \eqref{eq:b_4n+2}} & \ge i+4 \\
 1          & \ge 2  & 5 & \ge i+1\ \text{by \eqref{eq:8n+5}} & \ge i+3 \\
 5          & \ge 3  & 1 & \ge i\ \text{by \eqref{eq:4n+1}}   & \ge i+3 \\
 3          & \ge 2  & 3 & \ge i+2\ \text{by \eqref{eq:b_4n+3}} & \ge i+4 \\
 7          & \ge 2  & 7 & \ge i+2\ \text{by \eqref{eq:b_4n+3}} & \ge i+4
\end{array}
\]
In every case the contribution is $\equiv 0\pmod{2^{i+3}}$, so
$\overline{b}_{M-2}(8n+6)\equiv 0\pmod{2^{i+3}}$ and the result follows.

\subsection*{Acknowledgements}

The authors are partially supported by a Start-Up Grant from Ahmedabad University.


\begin{thebibliography}{NDMK24}

\bibitem[And98]{Andrews1998}
G.~E. Andrews.
\newblock {\em The theory of partitions}.
\newblock Cambridge Mathematical Library. Cambridge University Press, Cambridge, 1998.
\newblock Reprint of the 1976 original.

\bibitem[BS24]{BuragohainSaikia2024} P.~Buragohain and N.~Saikia.
\newblock Some new congruences for overcubic partitions with $r$-tuples.
\newblock {\em Arab. J. Math.}, 13(3):663-677, 2024.

\bibitem[CS26]{ChaconSellers}
D.~Chaćon and J.~A.~Sellers.
\newblock Congruences for Overcuic Partition $k$-tuples.
\newblock {\em arXiv preprint arXiv:2606.21780}, 2026.

\bibitem[Cha10a]{Chan1}
H.-C. Chan.
\newblock Ramanujan's cubic continued fraction and an analog of his ``most beautiful identity''.
\newblock {\em Int. J. Number Theory}, 6(3):673--680, 2010.

\bibitem[Cha10b]{Chan2}
H.-C. Chan.
\newblock Ramanujan's cubic continued fraction and {R}amanujan-type congruences for a certain partition function.
\newblock {\em Int. J. Number Theory}, 6(4):819--834, 2010.

\bibitem[CJY25]{Chen-Zing-Yao}
J.~Chen, J.~Jin, and O.~X.~M. Yao.
\newblock Proofs of two conjectures on infinite families of congruences of overcubic partition triples.
\newblock {\em Bull. Malays. Math. Sci. Soc.}, 48(6):Paper No.~210, 16 pp., 2025.

\bibitem[CL04]{Corteel2004}
S.~Corteel and J.~Lovejoy.
\newblock Overpartitions.
\newblock {\em Trans. Amer. Math. Soc.}, 356(4):1623--1635, 2004.


\bibitem[Hir17]{Hirschhorn}
M.~D. Hirschhorn.
\newblock {\em The power of {$q$}}, volume~49 of {\em Developments in Mathematics}.
\newblock Springer, Cham, 2017.
\newblock A personal journey, with a foreword by George E. Andrews.

\bibitem[Joh20]{Johnson}
W.~P.~Johnson.
\newblock {\em An introduction to $q$-analysis}.
\newblock American Mathematical Society, Providence, RI, 2020.

\bibitem[Kim10]{Kim}
B.~Kim.
\newblock The overcubic partition function mod $3$.
\newblock In {\em Ramanujan rediscovered}, volume~14 of {\em Ramanujan Math. Soc. Lect. Notes Ser.},
pages 157--163. Ramanujan Math. Soc., Mysore, 2010.

\bibitem[Kim12]{Kim2}
B.~Kim.
\newblock On partition congruences for overcubic partition pairs.
\newblock {\em Commun. Korean Math. Soc.}, 27(3):477--482, 2012.

\bibitem[Lin14]{Lin}
B.~L.~S.~Lin.
\newblock Arithmetic properties of overcubic partition pairs.
\newblock {\em Electron. J. Comb.}, 21(3):P3.35, 2014.

\bibitem[NS17]{NaikaShivashankar}
M.~S.~M.~Naika and C.~Shivashankar.
\newblock New congruences for overcubic partition pairs.
\newblock {\em Tbil. Math. J.}, 10(4):117--128, 2017.

\bibitem[NDMK24]{NayakaDharmendraKumar}
S.~S. Nayaka, B.~N. Dharmendra, and M.~C. Mahesh Kumar.
\newblock Divisibility properties for overcubic partition triples.
\newblock {\em Integers}, 24:Paper No.~A83, 2024.

\bibitem[Nay26]{Nayaka2026} S.~S. Nayaka. \newblock New modular congruences for overcubic partition triples. \newblock {\em Acta Univ. Sapientiae, Math.}, 18:21, 2026.


\bibitem[RB20]{RayBarman} C.~Ray and R.~Barman. \newblock Arithmetic properties of cubic and overcubic partition pairs. \newblock {\em Ramanujan J.}, 52(2):243--252, 2020.


\bibitem[SS25]{SS25}
M.~P. Saikia and A.~Sarma.
\newblock Further arithmetic properties of overcubic partition triples.
\newblock {\em Bull. Aust. Math. Soc.}, 112(2):260--273, 2025.

\bibitem[Sel14]{Sellers}
J.~A.~Sellers.
\newblock Elementary proofs of congruences for the cubic and overcubic partition functions. \newblock {\em Australas. J. Combin.}, 60(2):191--197, 2014.

\bibitem[SN18]{ShivashankarNaika}
C.~Shivashankar and M.~S.~M.~Naika.
\newblock New congruences for overcubic partition function.
\newblock {\em Mat. Vesn.}, 70(1):55--63, 2018.

\bibitem[Tan26]{Tang-BMMS}
D.~Tang.
\newblock Internal congruences modulo powers of 2 for overcubic partition triples.
\newblock {\em Bull. Malays. Math. Sci. Soc.}, 49(2):Paper No.~87, 2026.

\bibitem[ZZ11]{ZhaoZhong}
H.~Zhao and Z.~Zhong.
\newblock Ramanujan-type congruences for a partition function.
\newblock {\em Electron. J. Combin.}, 18(1):Paper No.~P58, 2011.

\end{thebibliography}
\end{document}